\documentclass[12pt]{amsart}
\usepackage{color}
\usepackage{verbatim}

\usepackage{eucal}
\begin{document}
\numberwithin{equation}{section}
\newtheorem{cor}{Corollary}[section]
\newtheorem{theorem}[cor]{Theorem}
\newtheorem{prop}[cor]{Proposition}
\newtheorem{lemma}[cor]{Lemma}
\newtheorem*{lemma*}{Lemma}
\theoremstyle{definition}
\newtheorem{defi}[cor]{Definition}
\theoremstyle{remark}
\newtheorem{remark}[cor]{Remark}
\newtheorem{example}[cor]{Example}

\newcommand{\cC}{\mathcal{C}}
\newcommand{\cL}{\mathcal{L}}
\newcommand{\cH}{\mathcal{H}}
\newcommand{\cP}{\mathcal{P}}
\newcommand{\cun}{\cC^{\infty}}
\newcommand{\cz}{{\mathbb C}}
\newcommand{\hz}{{\mathbb H}}
\newcommand{\rz}{{\mathbb R}}
\newcommand{\fm}{f^{-1}}
\renewcommand{\index}{\mathrm{index}}
\newcommand{\N}{\mathbb{N}}
\newcommand{\ov}{\widetilde}
\newcommand{\oD}{\overline{D}}
\newcommand{\px}{\partial_x}
\newcommand{\py}{\partial_y}
\newcommand{\R}{{\mathbb R}}
\newcommand{\Z}{{\mathbb Z}}
\newcommand{\C}{{\mathbb C}}
\newcommand{\SM}{{\mathbb S}}
\newcommand{\Spec}{\operatorname{Spec}}
\newcommand{\supp}{\mathrm{supp}}
\newcommand{\tf}{{\tilde{f}}}\newcommand{\tg}{{\tilde{g}}}
\newcommand{\tih}{\tilde{h}}
\newcommand{\tphi}{\tilde{\phi}}
\newcommand{\vol}{\mathrm{vol}}
\newcommand{\Xe}{X_\epsilon}
\newcommand{\zz}{{\mathbb Z}}

\def\f{\varphi}
\def\e{\varepsilon}
\def\d{\delta}
\def\O{\Omega}
\def\a{\alpha}
\def\b{\beta}
\def\g{\gamma}
\def\la{\langle}
\def\ra{\rangle}
\def\res{|}
\def\SO{{\rm SO}}
\def\Spin{{\rm Spin}}
\def\U{{\rm U}}
\def\Hol{{\rm Hol}}
\def\Ric{{\rm Ric}}
\def\End{{\rm End}}
\def\tr{{\rm tr}}
\def\Id{{\rm Id}}
\def\Scal{{\rm Scal}}
\def\vol{{\rm vol}}
\newcommand{\be}{\begin{equation}}
\newcommand{\ee}{\end{equation}}
\def\beq{\begin{eqnarray*}}
\def\eeq{\end{eqnarray*}}
\def\p{\psi}
\def\.{{\cdot}}
\def\n{\nabla}
\def\nm{\nabla^g}
\def\dt{{\partial_t}}
\def\dz{\partial_z}
\def\dzb{\partial_{\bar z}}

\newcommand{\D}{\mathbb{D}}
\newcommand{\pr}{\partial_r}

\title{Ricci surfaces}
\author{Andrei Moroianu}\thanks{Partially supported by the ANR-10-BLAN 0105 grant of the
Agence Nationale de la Recherche}
\address{Andrei Moroianu \\ Universit\'e de Versailles-St Quentin \\
Laboratoire de Math\'ema\-tiques \\ UMR 8100 du CNRS\\
45 avenue des \'Etats-Unis\\
78035 Versailles, France }
\email{andrei.moroianu@math.cnrs.fr}
\author{Sergiu Moroianu}\thanks{Partially supported by the LEA ``MathMode'' and the CNCS grant PN-II-RU-TE-2012-3-0492.}
\address{Sergiu Moroianu \\ Institutul de Matematic\u{a} al Academiei Rom\^{a}ne\\
P.O. Box 1-764\\RO-014700
Bu\-cha\-rest, Romania}
\email{moroianu@alum.mit.edu}
\date{\today}
\begin{abstract}
A Ricci surface is a Riemannian $2$-manifold $(M,g)$ whose Gaussian curvature $K$ satisfies 
$K\Delta K+g(dK,dK)+4K^3=0$. 
Every minimal surface isometrically embedded in $\R^3$ is a Ricci surface of
non-positive curvature.
At the end of the 19$^{\text{th}}$ century Ricci-Curbastro has proved that conversely, every point $x$ of a Ricci surface
has a neighborhood which embeds isometrically in $\R^3$ as a minimal 
surface, provided $K(x)<0$. We prove this result in full generality by showing that Ricci surfaces can be locally
isometrically embedded either minimally in $\R^3$ or maximally in $\rz^{2,1}$, including near
points of vanishing curvature. We then develop the theory of closed Ricci surfaces, possibly with
conical singularities, and construct classes of examples in all genera $g\ge 2$.
\end{abstract}

%\subjclass[2010]{49Q05,53C27,53C42}
\maketitle

\section{Introduction}

In 1873 Ludwig Schl\"afli asked the following question, still unanswered today ({\em cf.} \cite{yau1}, \cite{yau2}):

\begin{quote}\em{Can every Riemannian surface $(M^2,g)$ be
locally isometrically embedded in the flat space $\R^3$?}
\end{quote}

The problem reduces to a non-linear equation of Monge-Amp\`ere type. This equation can be easily
solved near points
$x$ where the Gaussian curvature $K(x)$ is non-vanishing, but is degenerate at points where
$K(x)=0$. A partial positive answer
was recently obtained provided that the gradient of the Gaussian curvature has a special behavior
in the neighborhood of the zero set of $K$ (see \cite{hhl} and references therein).

A related question was asked in 1895 by Gregorio Ricci-Curbastro \cite{ricci} about minimal
embeddings in $\R^3$:

\begin{quote}
{\em When does a Riemannian surface $(M^2,g)$ carry minimal
local isometric embeddings in the flat space $\R^3$?}
\end{quote}

The answer is known near points of non-zero Gaussian curvature:

\begin{theorem}[\cite{ricci}, \cite{blaschke}, p. 124]\label{tr}
A Riemannian surface $(M^2,g)$ with negative Gaussian curvature $K<0$
has local isometric embeddings as minimal surface in the flat space $\R^3$
if and only if one of the two equivalent conditions below holds:
\begin{itemize}
\item[(i)] The metric $\sqrt{-K}g$ is flat.
\item[(ii)] The Gaussian curvature satisfies
\be\label{R}
K\Delta K+g(dK,dK)+4K^3=0
\ee
where $\Delta =\delta^gd$ denotes the scalar Laplace operator of the metric $g$.
\end{itemize}
\end{theorem}

Condition (i) is usually referred to as the {\em Ricci condition}. This condition does not hold in
general for minimal surfaces in $\R^n$ for $n\ge 4$, see \cite{lawson}, \cite{pinl}. In \cite{vlachos2}, Vlachos obtains some necessary
conditions for the existence of local minimal immersions of $(M^n,g)$ in $\R^{n+p}$ for all $n\ge
2$ and $p\ge 1$. A generalization of Theorem \ref{tr} to pluriharmonic submersions of K\"ahler
manifolds $(M^{2n},g,J)$ in $\R^{2n+1}$ was obtained by Furuhata \cite{furuhata}.

Our main result is the extension of Theorem \ref{tr} to the general case, with no assumption on the
Gaussian curvature. Of course, the Ricci condition (i) no longer makes sense at points where $K$ vanishes, but we
simply use the Ricci condition (ii) instead. It turns out that if $K$ satisfies \eqref{R}, then it
either vanishes identically, or does not change sign on $M$. Both signs
might appear, and they correspond to minimal immersions in the Euclidean space, respectively to
maximal immersions in the Lorentz space:

\begin{theorem}\label{main}
Let $(M^2,g)$ be a connected Riemannian surface whose Gaussian curvature $K$ satisfies \eqref{R}.
Then $K$ does not change its sign on $M$. If $K\leq 0$, then $M$ can be locally
isometrically immersed in $\R^3$ as a minimal surface. If $K\geq 0$, then $M$ can be locally
isometrically immersed in the Lorentz space $\R^{2,1}$ as a maximal surface.
\end{theorem}

Using the spinorial characterization of isometric embeddings of surfaces in $\R^3$ or $\R^{2,1}$, 
a significant step in the proof of this theorem reduces to a statement formulated only in terms of
holomorphic and harmonic functions on $\cz$.

\begin{theorem}\label{th2}
Let $\Omega\subset\cz$ be a simply connected domain and $F\in\cun(\Omega,\rz)$.
Assume that $\log|F|$ is harmonic at every point where $F\neq 0$. Then $F$ does not change sign, 
and there exists a holomorphic function $h$ with $|F|=|h|^2$.
\end{theorem}
The main difficulty in Theorem \ref{th2} is to show that the zeros of $F$ are isolated. 
This is accomplished in Theorem \ref{is} below, using ideas stemming from potential theory. We
assume that the function $F$ has some non-isolated zeros. We prove in Lemma \ref{s} that at its
non-isolated zeros $F$ vanishes to infinite order. Then we show that the connected components of the
complement of the set of non-isolated zeros cannot be simply connected, so there exist simple closed
curves avoiding the zero set of $F$ and confining some non-isolated zeros of $F$. For every simple
closed curve $\gamma$ on which $F$ does not vanish, we define a ``virtual measure" of the zero set
of $F$ lying in the region $\Omega$ bounded by $\gamma$. For example, in the case where $\log(F)$ is
defined by convolution of the Green kernel of the Laplacian with a measure $\mu$ supported on some
compact set $C$, the zero set of $F$ is $C$ and the virtual measure of $\Omega$ is just $\mu(C)$.
The main properties of the virtual measure are positivity (Lemma \ref{tw1}) and additivity. 
To obtain a contradiction we divide $F$ by a sufficiently large power $n$ of the distance function
to a non-isolated zero. We obtain again a smooth non-negative function whose logarithm is harmonic outside the zero-set and whose virtual measure
decreases by $2\pi n$ compared to that of $F$, thus contradicting the positivity of the virtual measure.
This proof is carried out in detail in Section \ref{zslh}.

In a second part of the paper we study the existence and uniqueness question, in a given conformal
class, of metrics satisfying the Ricci condition, also called {\em Ricci metrics}.
We construct Ricci metrics of non-positive (resp.\ non-negative) curvature from spherical (resp.\
hyperbolic) metrics with conical points of angles integer multiples of $2\pi$. For non-positive
curvature we get for instance that every hyperelliptic surface of odd genus admits a Ricci metric.
In the non-negative case, on a closed surface there exist conical Ricci metrics of positive
curvature with prescribed conical singularities. These results are grouped in Section \ref{ss}. 

For the convenience of those readers more familiar with 
the classical viewpoint of minimal surfaces, 
we describe in the Appendix the link between our approach 
and the standard Weierstrass-Enneper representation.

The theory of minimal surfaces, although more than two centuries old, is still a
very active field of research, and it is somehow surprising that the intrinsic characterization of
minimal surfaces in $\R^3$ obtained here was only available so far in the case of non-vanishing
Gaussian curvature.
For the analytical aspects of minimal surfaces we refer to the recent monograph by Colding and Minicozzi \cite{c-m}.
From the huge literature in the subject, we would like to single out Taubes' recent study
\cite{taubes} of the moduli space of minimal surfaces embedded in $\hz^3$, and
Weber and Wolf's construction \cite{ww} of embedded minimal surfaces in $\R^3$ using the notion of {\em orthodisks}, which seems
to be somewhat related to our method of constructing compact Ricci surfaces in Section \ref{ss} below.

{\sc Acknowledgments.} We have benefited from many enlightening discussions with Christophe
Margerin.
His suggestions coming from potential theory inspired us the key ideas used in the proof of Theorem \ref{is}.

\section{Preliminaries}

\subsection{Conformal metric changes on surfaces}
We start by recalling some well-known facts in conformal geometry. Assume that $g_0$ and $g:=e^{-2f}g_0$
are Riemannian metrics on a surface $M$. Let $\Delta=\delta^{g}d$ and $K$, respectively $\Delta_0=\delta^{g_0}d$ and $K_0$,
denote the Laplacian and the 
Gaussian curvature of $g$ and $g_0$. Then the following formulas hold ({\em cf.} \cite[p.\ 59]{besse}):
\begin{align}
\label{1} &\Delta=e^{2f}\Delta_0,\\
\label{3} &K=e^{2f}(K_0-\Delta_0 f).
\intertext{If we fix a spin structure and denote by $D$ and $D_0$ the Dirac operators corresponding to 
$g$ and $g_0$ respectively, then Hitchin's classical conformal covariance relation reads}
\label{2} &D\psi=e^{\frac{3f}2}D_0(e^{-\frac f2}\psi).
\end{align}

\subsection{Ricci surfaces} Motivated by Ricci-Curbastro's local characterization of minimal
surfaces in $\R^3$ (Theorem \ref{tr}), we make the following:
\begin{defi}\label{ricsur}
A Riemannian surface $(M,g)$ whose Gaussian curvature $K$ satisfies the identity \eqref{R}
\[K\Delta K+g(dK,dK)+4K^3=0\]
is called a {\em Ricci surface}, and $g$ is called a \emph{Ricci metric}.
\end{defi}

As mentioned in the introduction, Ricci metrics have several nice characterizations near points
where the Gaussian curvature is negative:
\begin{lemma}\label{equiv}
Let $(M,g)$ be a Riemannian surface with negative curvature $K<0$. The following four conditions are equivalent:
\begin{itemize}
\item $g$ is a Ricci metric;
\item $\Delta\log(-K)+4K=0$;
\item the metric $(-K)^{1/2}g$ is flat;
\item the metric $(-K)g$ is spherical, {\em i.e.}, of constant Gaussian curvature $1$.
\end{itemize}
\end{lemma}
\begin{proof}
We compute directly $\Delta\log(-K)+4K=K^{-2}(K\Delta K+|dK|^2+4K^3)$, hence the first two conditions are equivalent.
For $r\in\rz$ set $g_r:=(-K)^rg$. By Eq.\ \eqref{3}, the Gaussian curvature $K_r$ of the metric $g_r$
equals
\begin{align}\label{Kr} K_r=(-K)^{-r}\left(K+\tfrac12\Delta (\log(-K)^r)\right).
\end{align}
Assuming $\Delta\log(-K)+4K=0$ we get $K_r=(1-2r) (-K)^{-r}K$, hence $g_{1/2}$
is flat and $g_1$ has constant Gaussian curvature equal to
$1$. Conversely, if $K_{1/2}=0$ then \eqref{Kr} for $r=\frac12$ implies $\Delta\log(-K)+4K=0$, and
the same conclusion holds if $K_1=1$.
\end{proof}
We thus see that the conformal class of a negatively curved Ricci metric contains both a flat and a
round metric. Conversely, we can construct Ricci metrics in any conformal class known \emph{a
priori} to contain both a spherical and a flat metric:

\begin{lemma}\label{trfsr}
Let $g_{1/2}$ be a flat metric on a surface $M$ and $V\in\cun(M)$, $V>0$ such that $g_1:=Vg_{1/2}$
is spherical. Then $g:=V^{-1}g_{1/2}$ is a Ricci metric of curvature $-V^2$.
\end{lemma}
\begin{proof}
Denote by $K,K_{1/2},K_1$ and $\Delta,\Delta_{1/2},\Delta_1$ the Gaussian curvatures and the
Laplacians of $g,g_{1/2}$, resp.\ $g_1$. From \eqref{3},
\begin{align}
\label{ku} &K_1=V^{-1}\Delta_{1/2}(\tfrac12\log V),\\
\label{kv} &K=V\Delta_{1/2}(-\tfrac12\log V).
\end{align}
From \eqref{ku}, since $g_1$ is spherical, we get $\Delta_{1/2}\log V=2V$ and so from \eqref{kv}
$K=-V^2$. Therefore
\[\Delta \log (-K)=2V\Delta_{1/2}\log V=4V^2=-4K,\]
hence $g$ is a Ricci metric by Lemma \ref{equiv}.
\end{proof}

The corresponding statements in positive curvature are similar and left to the reader:
\begin{lemma}\label{kp}
Let $(M,g)$ be a Riemannian surface of positive curvature $K>0$. The following conditions are equivalent:
\begin{itemize}
\item $g$ is a Ricci metric;
\item $\Delta\log K+4K=0$;
\item the metric $K^{1/2}g$ is flat;
\item the metric $Kg$ is hyperbolic, {\em i.e.}, of constant Gaussian curvature $-1$.
\end{itemize}
\end{lemma}

\begin{lemma}\label{hfr}
Let $g_{1/2}$ be a flat metric on a surface $M$ and $V\in\cun(M)$, $V>0$ such that $g_1:=Vg_{1/2}$
is hyperbolic. Then $g:=V^{-1}g_{1/2}$ is a Ricci metric of Gaussian curvature $V^2$.
\end{lemma}

\section{Spinorial characterization of isometric embeddings in $\R^3$}

In \cite{friedrich} Friedrich remarked that 
local isometric embeddings of a Riemannian surface in the Euclidean space $\R^3$ are characterized by
special spinor fields on the surface called {\em generalized Killing spinors} (see also \cite{bgm}):

\begin{lemma}[\cite{friedrich}, Thm. 13]\label{l4}
Let $W$ be a symmetric tensor on the spin surface $(M^2,g)$.  There exists a locally isometric embedding 
$(M,g)\to\R^3$ with Weingarten tensor $W$
if and only if $M$ carries a non-zero spinor $\psi$ satisfying
\begin{align}\label{gks}
\nabla_X\p=\tfrac12W(X)\.\p,&&(\forall)\ X\in TM.
\end{align}
\end{lemma}

Moreover, due to the algebraic structure of spinors in two dimensions, a generalized Killing spinor
can be characterized by a seemingly weaker condition:

\begin{lemma}\label{l41}
A non-zero spinor $\psi$ on a Riemannian surface $(M^2,g)$ satisfies \eqref{gks} for some symmetric
tensor $W$ if and only if it has constant length and there exists a real function $w$ such that
$D\psi=w\psi$. In this situation, $w=-\frac12\tr(W).$
\end{lemma}
\begin{proof}
Assume first that $\psi$ satisfies \eqref{gks}. Taking the Clifford contraction in this equation
yields $D\psi=-\tfrac12 \tr(W)\psi$. Moreover $d|\psi|^2(X)=2\langle\nabla_X\psi,\psi\rangle=\langle
W(X)\.\psi,\psi\rangle=0$ for every tangent vector $X$, so $\psi$ has constant length.

Conversely, assume that $\psi$ has constant length and $D\psi=w\psi$. Since $\psi$ is non-zero, one
may assume that $|\psi|=1$.
Let $e_1,e_2$ be a local orthonormal basis of the tangent bundle. For dimensional reasons, the
spinors $\psi$, $e_1\.\psi,\ e_2\.\psi$ and $e_1\.e_2\.\psi$
define a local orthonormal basis (over $\R$) of the spin bundle.
Since $\langle\nabla_X\psi,\psi\rangle=\tfrac12d|\psi|^2(X)=0$ for every tangent vector $X$,
there exist an endomorphism field $A$ of $TM$ and a 1-form $a$ such that 
\begin{align}\label{r}
\nabla_X\psi=A(X)\.\psi+a(X)e_1\.e_2\.\psi
\end{align}
for every $X\in TM$. Let $(a_{ij})$ be the matrix of $A$ in the basis $e_1,e_2$ and $a=a_1e_1^*+a_2e_2^*$.
After Clifford contraction, \eqref{r} yields 
$$w\psi=D\psi=-(a_{11}+a_{22})\psi+(a_{12}-a_{21})e_1\.e_2\.\psi+(a_2 e_1-a_1e_2)\.\psi.$$
Using again that $\psi$, $e_1\.\psi,\ e_2\.\psi$ and $e_1\.e_2\.\psi$ are linearly independent over $\R$, we get
$w+\tr(A)=0,\ a_{12}=a_{21}$ and $a=0$. Thus \eqref{r} is equivalent to \eqref{gks} for $W=2A$.
\end{proof}

Specializing to the case of minimal surfaces, we get:

\begin{cor}\label{l5}
A Riemannian surface has local isometric minimal embeddings in $\R^3$ if and only if it carries
local non-zero harmonic spinors of constant length.
\end{cor}

This provides a simple characterization of metrics which embed locally as minimal surfaces in $\R^3$,
in terms of the conformal factor of the metric in isothermal coordinates.

\begin{cor}\label{l6}
Let $g_0=dx^2+dy^2$ be the flat metric on some domain $\Omega\subset\C$, and $f:\Omega\to \R$ any smooth function.
The metric $g=e^{-2f}g_0$  has (locally) an isometric embedding in $\R^3$
as minimal surface if and only if near every $x\in\Omega$ there exists a pair of 
holomorphic functions $(a,b)$ such that $e^{-f}=|a|^2+|b|^2$.
\end{cor}
\begin{proof}
Let $x\in\Omega$ and assume that some $U\subset\Omega$, $x\in\U$, has a local isometric embedding
as minimal surface in $\R^3$. The previous corollary shows the existence of
a harmonic spinor $\psi$ of unit length with respect to $g$ defined on some open set $V\subset U$,
$x\in V$. By \eqref{2}, $e^{-f/2}\psi$ is
a harmonic spinor on $(V,g_0)$ of square length $e^{-f}$. The (complex) spin bundle of
$(\Omega,g_0)$ is trivial
and spanned by two parallel spinors $\psi^\pm\in \cun(\Sigma^\pm \Omega)$. Write
$\psi=a\psi^++\bar b\psi^-$ for some complex-valued functions $a,b$ on $V$.
Since 
$D_0=\begin{bmatrix}0&-\partial_z\\ 
\partial_{\bar z}&0\end{bmatrix}$
with respect to the
basis
$\{\psi^+,\psi^-\}$, $D_0\psi=0$ is equivalent to $a$ and $b$ being holomorphic. The converse
statement is similar.
\end{proof}

We consider now the case of 
Riemannian surfaces $(M^2,g)$ (locally) isometrically embedded as space-like surfaces in the
Lorentz space $\R^{2,1}$.
The restriction of the (complex) spin bundle $\Sigma\R^{2,1}$ to $M$ can be identified with the
spin bundle $\Sigma M=\Sigma^+M\oplus\Sigma^-M$ of $(M,g)$. With respect to this identification, the
Clifford action of the time-like normal vector $\nu$ of square norm $-1$ is given by
$\nu\.\psi=\bar\psi:=\psi^+-\psi^-$ and
the natural (indefinite) Hermitian product $h$ on $\Sigma\R^{2,1}$ corresponds to 
$h(\psi,\psi):=|\psi^+|^2-|\psi^-|^2$ for $\psi=\psi^++\psi^-$. The restriction $\psi$ of a
parallel spinor from $\Sigma\R^{2,1}$ to $\Sigma M$
satisfies $\nabla_X\p=\tfrac12W(X)\.\bar\p$. The arguments from the previous subsection remain
valid {\em mutatis mutandis} and we obtain the following characterization of maximal embeddings in
the Lorentz space:

\begin{lemma}\label{l61}
Let $g_0=dx^2+dy^2$ be the flat metric on some domain $\Omega\subset\C$ and $f:\Omega\to \R$ any
smooth function. The metric $g=e^{-2f}g_0$ admits local isometric embeddings in $\R^{2,1}$
as maximal surface if and only if locally on $\Omega$ there exist pairs of holomorphic functions
$(a,b)$ such that $e^{-f}=|a|^2-|b|^2$.
\end{lemma}

As a corollary we recover the well-known property that the curvature of minimal surfaces in $\R^3$ 
and of maximal surfaces in $\R^{2,1}$ satisfies the Ricci condition.

\begin{lemma}\label{l7}
Let $(M^2,g)\subset \R^3$ be either a minimal surface in $\R^3$ or a maximal surface in the Lorentz space $\R^{2,1}$.
Then the Gaussian curvature $K$ of $M$ satisfies the Ricci condition \eqref{R}, namely
$K\Delta K+g(dK,dK)+4K^3=0$.
\end{lemma}
\begin{proof}
We write $g=e^{-2f}g_0$ where $g_0$ is flat. By \eqref{3} we have 
\be\label{k}K=-e^{2f}\Delta_0f.\ee
Moreover, Corollary \ref{l6} and Lemma \ref{l61} show that $e^{-f}=|a|^2+\e |b|^2$ for some
holomorphic functions $a,b$.
Here the sign $\e$ is $1$ if we work in $\rz^3$ and $-1$ in Lorentz space.
The Laplacian $\Delta_0$ of the flat metric $g_0:=dx^2+dy^2=|dz|^2$ on $\R^2$ satisfies
\be\label{4}\Delta_0=-4\dz\dzb,
\ee
where
\begin{align*}
\dz:=\tfrac{\partial}{\partial z}=\tfrac12\left(\tfrac{\partial}{\partial
x}-i\tfrac{\partial}{\partial
y}\right),&&
\dzb:=\tfrac{\partial}{\partial \bar{z}}=\tfrac12\left(\tfrac{\partial}{\partial
x}+i\tfrac{\partial}{\partial
y}\right).
\end{align*}
Using \eqref{4} we infer
\begin{align*}
\tfrac14 \Delta_0f=&-\dz\dzb f=-\dz\dzb\log(|a|^2+\e|b|^2)=\dz\left(\frac{a\overline{a'}+\e
b\overline{b'}}{|a|^2+\e |b|^2}\right)\\
=&\frac{(a'\overline{a'}+\e b'\overline{b'})(|a|^2+\e |b|^2)-
(a\overline{a'}+\e b\overline{b'})(a'\overline{a}+\e b'\overline{b})}{(|a|^2+\e|b|^2)^2}\\
=&\e \frac{|a{b'}-b{a'}|^2}{(|a|^2+\e|b|^2)^2}.
\end{align*}
We thus obtain $e^{-2f}\Delta_0f=4\e|a{b'}-b{a'}|^2$, and since the logarithm of the norm of a
non-vanishing holomorphic function
is harmonic, we get from \eqref{1} and \eqref{k} at points where $K$ does not vanish
\begin{align*} 0=&\Delta_0(\log|e^{-2f}\Delta_0f|)=\Delta_0(\log|e^{2f}\Delta_0f|-4f)=\Delta_0(\log|K|)-4\Delta_0 f\\
=&\Delta_0(\log|K|)+4e^{-2f} K=
e^{-2f} (\Delta (\log|K|)+4K)
\end{align*}
whence $\Delta (\log|K|)+4K=0$. Using the formula $\Delta=\delta^gd$ we obtain
$$-4K=\Delta (\log|K|)=\delta^g\left(\frac{dK}{K}\right)=\frac{\Delta K}{K}+\frac{g(dK,dK)}{K^2},$$
which is equivalent to \eqref{R}.
\end{proof}

Let us remark that there is a close link, already noted by Al\'\i as \cite{alias}, between minimal
surfaces in $\R^3$ and maximal surfaces in $\R^{2,1}$. In our setting, this duality is obtained by 
associating to any Ricci
metric of the form $(|a|^2+|b|^2)^2|dz|^2$, which by Corollary \ref{l6} embeds as minimal surface in
$\R^3$, the Ricci metric $(|a|^2-|b|^2)^2|dz|^2$ which by Lemma \ref{l61} embeds as maximal surface
in $\R^{2,1}$. This correspondence is not intrinsic since it depends on the choice of the
holomorphic functions $a$ and $b$ representing the conformal factor $e^{-f}=|a|^2+|b|^2$.

\section{Log-harmonic and holomorphic functions} \label{zslh}

In this section we prove Theorem \ref{th2}, which is one of the central results of this paper. 

\begin{defi}
A real-valued function $F$ defined on some open set $\Omega\subset \C$
is called {\em log-harmonic} if $F\in\cun(\Omega,\R)$ and $\log |F|$ is harmonic on the open set where
$F\ne 0$.
\end{defi}

It is clear that for every holomorphic function $h$ on some $\Omega\subset \cz$, its square norm $F:=|h|^2$ is
log-harmonic. Conversely, if $F$ is log-harmonic, and, say, non-negative,
does there exist a holomorphic function $h$ on $\Omega$ such that $F=|h|^2$? If $F>0$ and $\Omega$
is simply connected, the answer is standard:

\begin{lemma}\label{l1} Let $F>0$ be a positive log-harmonic function on some
simply connected domain $\Omega\subset\C$. Then there exists a holomorphic function $h$ on $\Omega$
such that $F=|h|^2$.
\end{lemma}
\begin{proof}
By \eqref{4} $\dz \log F$ is holomorphic on $\Omega$ and since $\Omega$ is simply connected, 
there exists a holomorphic function $g_1$ on $\Omega$ with $\dz \log F=\dz g_1$. Thus
$g_2:=\overline{\log F-g_1}$
is holomorphic on $\Omega$ and $\log F=g_1+\overline{g_2}$. Since $\log F$ is real, we have
$\log F=\Re(g_1+g_2)$,
so $F=|h|^2$ for $h:=e^{\tfrac{g_1+g_2}2}$.
\end{proof}

The question, answered by Theorem \ref{th2}, is whether the local solutions can be extended
globally including on the zero set. We need some preliminary results first.

\begin{lemma}\label{n}
If two holomorphic non-vanishing functions $h_1$ and $h_2$ have the same norm on a connected open
subset $\Omega\subset \C$ then
there exists $\theta\in[0,2\pi)$ such that $h_1=e^{i\theta}h_2$ on $\Omega$.
\end{lemma}
\begin{proof}
Clear from the maximum principle applied to $h_1/h_2$.
\end{proof}

For a vector $G=(G_1,G_2)\in\cz^2$ we denote $|G|^2:=|G_1|^2+|G_2|^2$ and $|G|^2_-:= |G_1|^2-|G_2|^2$.

\begin{lemma}\label{l3}
Let $\Omega\subset\C$ be a connected domain and $G=(G_1,G_2):\Omega\to\C^2\setminus\{(0,0)\}$ be a
holomorphic map. If $|G(0)|^2>0$
then another holomorphic map $H:\Omega\to\C^2$ satisfies $|G|^2=|H|^2$ if and only if there exists
$A\in\U(2)$ such that $H=AG$. Similarly, if $|G(0)|^2_->0$
then another holomorphic map $H:\Omega\to\C^2$ satisfies $|G|^2_-=|H|^2_-$ if and only if there
exists
$A\in\U(1,1)$ such that $H=AG$. Moreover, if $\Delta_0\log |G|^2$, respectively
$\Delta_0\log|G|^2_-$ are not identically $0$, then the matrix $A$ is unique.
\end{lemma}
\begin{proof}
 The ``if'' part is obvious. Assume now that $|G|^2=|H|^2$. Since $G(0)\ne 0$, one of its components, say $G_1$ does
not vanish at $0$, and thus on some smaller domain $\Omega'\subset \Omega$. The three functions
$a:=G_2/G_1$,
$b:=H_1/G_1$ and $c:=H_2/G_1$
are holomorphic on $\Omega'$ and satisfy
\be\label{1a}1+|a|^2=|b|^2+|c|^2.\ee 
Taking the double derivative $\dz\dzb$ ({\em i.e.}, $-\frac14\Delta_0$) in this relation yields $|a'|^2=|b'|^2+|c'|^2.$
If $a'\equiv 0$ on $\Omega'$ then $a,b,c$ are constant, hence $G_2, H_1, H_2$ are constant multiples
of
$G_1$ and the conclusion follows from the transitivity of the action of $U(2)$ on the unit sphere
$\SM^3$. Assume that $a'$ does not vanish on some disc $\Omega''\subset \Omega'$. The
holomorphic functions $\alpha:=b'/a'$
and $\beta:=c'/a'$ satisfy
\be\label{2a}|\alpha|^2+|\beta|^2=1\ee
on $\Omega''$. Differentiating again with respect to $\dz\dzb$ we get 
$|\alpha'|^2+|\beta'|^2=0$, so $\alpha$ and $\beta$ are constant on $\Omega''$, and thus on
$\Omega'$. We
then have 
$b'=\alpha a'$ and $c'=\beta a'$ on $\Omega'$, so there exist constants $\gamma$ and $\delta$ such
that
$b=\alpha a+\gamma$ and
$c=\beta a +\delta$ on $\Omega'$. This reads $H=AG$ on $\Omega'$, thus on $\Omega$, where 
$A=\begin{bmatrix}\gamma&\alpha\\
\delta&\beta \end{bmatrix}$.
It remains to check that $A\in\U(2)$. From \eqref{1a} we get 
$$1+|a|^2=|\alpha a+\gamma|^2+|\beta a +\delta|^2=|a|^2+(|\gamma|^2+|\delta|^2)+
2\Re(a(\alpha\bar\gamma+\beta\bar\delta)),$$
so the imaginary part of the holomorphic function $a(\alpha\bar\gamma+\beta\bar\delta)$ 
vanishes. Since $a$ is non-constant (see above)
we deduce that $\alpha\bar\gamma+\beta\bar\delta=0$ and $|\gamma|^2+|\delta|^2=1$.
Together with \eqref{2a}, this shows that $A\in\U(2)$.

In the semi-definite case the proof proceeds similarly with the same notation: we have
$1-|a|^2=|b|^2-|c|^2$ hence $|a'|^2=|c'|^2-|b'|^2$. If $a'=0$ then $G_2=\gamma G_1$ for some
constant $\gamma$ and by Lemma \ref{n} $b'=e^{i\theta} c'$, which implies easily that $H_1,H_2$ are
constant multiples of $G_1$ and the conclusion follows. If $a'\neq 0$ the functions $\alpha:=b'/a'$
and $\beta:=c'/a'$ satisfy $|\alpha|^2-|\beta|^2=1$ so $|\alpha'|^2-|\beta'|^2=0$. These two
identities imply easily that $\alpha,\beta$ are constants. The rest of the proof is unchanged.
\end{proof}

Let now $F$ be a log-harmonic function. By Lemma \ref{l1}
every point where $F$ is non-zero has an open neighborhood on which there exists a holomorphic
function $h$ with $F=|h|^2$. 
The case of isolated zeros is only slightly more involved.

\begin{lemma}\label{l2}
Let $F:\D\to\R$ be a smooth non-negative function on the unit disc $\D\subset \C$ such that $F$
does not vanish on $\D^*:=\D\setminus\{0\}$. If $\log(F)$ is harmonic on $\D^*$,
then there exists a holomorphic function $h$ on $\D$ such that $F=|h|^2$.
\end{lemma}
\begin{proof}
We identify the universal cover $\widetilde{\D^*}$ of $\D^*$ with $\{z\in\cz;\Re(z)<0\}$ and the
projection from $\widetilde{\D^*}$ to $\D^*$ with the exponential map. The function 
$z\mapsto\log(F(e^z))$ is harmonic on $\widetilde\D^*$, so by Lemma \ref{l1}
there exists a holomorphic function $G$ on $\widetilde{\D^*}$ with $F(e^z)=|G(z)|^2$ for all $z\in \widetilde{\D^*}$.
By Lemma \ref{n}, there exists $\theta\in[0,1)$ such that $G(z+2\pi i)=e^{-2\pi i\theta}G(z)$. The function 
$$H(z):=e^{z\theta}G(z)$$
is thus invariant by translation with $2\pi i$, hence it descends to a holomorphic function $h$
on $\D^*$ with
$h(e^z)=H(z)$. Denoting $w:=e^z$ we get
$$F(w)=|G(z)|^2=|e^{-2z\theta}H(z)|^2=|w|^{-2\theta}|h(w)|^2.$$
This shows in particular that the function $h$ is bounded near the origin, 
so it extends to a holomorphic function $h$ on $\D$. 

Let $k$ be the vanishing order of $h$ at $0$. One has $h(w)=h_1(w)w^k$ with $h_1$ holomorphic and $h_1(0)\ne 0$.
Since $w\mapsto |w|^{2k-2\theta}=F/|h_1|^2$ is smooth near $0$, the exponent $2k-2\theta$ is an even integer.
By the choice of $\theta$ in $[0,1)$ we get $\theta=0$, so $F=|h|^2$ as claimed.
\end{proof}

The main difficulty in Theorem \ref{th2} is to show that a log-harmonic function does not have non-isolated zeros.

\begin{theorem}\label{is}
Let $F:\D\to\R$ be a smooth function on the unit disc $\D\subset \C$ such that 
$\log(|F|)$ is harmonic on $\D\setminus F^{-1}(\{0\})$. Then either $F$ vanishes identically, or
$F^{-1}(\{0\})$ is a discrete set. In particular, log-harmonic functions on $\D$ have constant sign.
\end{theorem}

\begin{proof}
The proof will be divided in several steps.

\begin{lemma}\label{s} Let  $z_0\in \D$ be a non-isolated zero of a log-harmonic function $F$. 
Then $F$ vanishes at infinite order at $z_0$. 
\end{lemma}

\begin{proof} On the open set $\D\setminus F^{-1}(\{0\})$ the function $\log(|F|)$ is harmonic, thus
$$0=\Delta_0(\log(|F|))=\delta^0(dF/F)=\Delta_0(F)/F+|dF|^2/F^2,$$
therefore
\be\label{lh}
F\Delta_0(F)+|dF|^2=0.
\ee
By restricting to a small disc centered at $z_0$ and composing $F$ with a translation one may take
$z_0=0$.
Assume that $F$ does not vanish at infinite order at $0$ and let $P$ be the principal part of $F$ near $0$.
Then $P$ is a homogeneous polynomial in $x$ , $y$ such that $F-P=o(r^n)$, where 
$r:=\sqrt{x^2+y^2}$ and $n$ is the degree of $P$. Clearly $\Delta_0(F)=\Delta_0(P)+o(r^{n-2})$ and
$dF=dP+o(r^{n-1})$.
From \eqref{lh} we get $P\Delta_0(P)+|dP|^2=o(r^{2n-2})$. On the other hand the left-hand side 
in this
equality is a homogeneous polynomial in $x$, $y$ of degree $2n-2$, thus showing that
\be\label{lh1}
P\Delta_0(P)+|dP|^2=0.
\ee
In polar coordinates we can write $P=r^nQ(\theta)$, where $Q(\theta)=P(\cos\theta,\sin\theta)$ 
is a trigonometric polynomial with real coefficients. Using the formulae
\begin{align*}dx^2+dy^2=dr^2+r^2d\theta^2,&&
\Delta_0=-\left(\tfrac1r\tfrac\partial{\partial r} +\tfrac{\partial^2}{\partial
r^2}+\tfrac1{r^2}\tfrac{\partial^2}{\partial \theta^2}\right),\end{align*}
equation \eqref{lh1} becomes
$$-r^nQ[nr^{n-2}Q+n(n-1)r^{n-2}Q+r^{n-2}Q'']+n^2r^{2n-2}Q^2+r^{2n-2}{Q'}^2=0$$
{\em i.e.}, ${Q'}^2=QQ''$. The solutions of this differential equation are $Q(\theta)=ae^{b\theta}$
for $a,b\in\R$.
Since $Q$ is a trigonometric polynomial, we necessarily have $b=0$ and thus $Q$ is constant.
Therefore $P(x,y)=a(x^2+y^2)^{n/2}$ and $a\ne 0$ by the assumption that $P\ne 0$. Incidentally this implies that $n$ is even,
but we do not need this observation. More importantly, since $F(z)=P(z)+o(|z|^n)=|z|^n(a+o(1))$, it turns out 
that $0$ is an isolated zero of $F$, contradicting the hypothesis. This proves the lemma.
\end{proof}

Let $Z$ denote the (closed) set of non-isolated zeros of $F$.  Assume that $F$ does not vanish identically on $\D$
and let $E$ denote a connected component of the open set $\D\setminus Z$.
By changing the sign of $F$ if necessary, we can assume that $F$ is non-negative on $E$.

If $E$ is simply connected, by Lemma \ref{l1} we can construct a holomorphic function $h$ on $E$
such that $F=|h|^2$. Since by Lemma \ref{s} $F$ must vanish to infinite order at every point of $Z$,
the function $F\chi_E$ is smooth on $\D$, where $\chi_E$ is the characteristic function of $E$.
Moreover, for every $z_0\in \D\setminus E$ we have $F(z)\chi_E(z)=o(|z-z_0|)$.

Extend the holomorphic function $h^2$ from $E$ to $\D$ by setting it to be $0$ on $\D\setminus E$.
At a point $z_0\in \D\setminus E$ we have
\[\frac{|h^2(z)-h^2(z_0)|}{|z-z_0|}=\frac{F(z)\chi_E(z)}{|z-z_0|}\]
tends to $0$ as $z\to z_0$. Therefore $h^2$ is holomorphic on $\D$, and thus its zeros are
isolated, which is the conclusion of Theorem \ref{is}.

We are left with the case where there are no simply connected components of $\D\setminus Z$.
Thus, we may assume that $E$ is not simply connected, hence we can find
a smooth simple curve $C$ in $E$ containing at least one non-isolated zero of $F$
in its interior. By slightly deforming $C$ if necessary, we can assume
that $C$ avoids also the isolated zeros of $F$, {\em i.e.}, $F$ does not vanish on $C$.
Using the Riemann uniformization theorem, we can identify the interior of $C$
with the unit disk $\D$. We can thus from now on assume that $F:\bar \D\to \R$ is smooth,
non-negative, has at least one non-isolated zero, is log-harmonic outside its zero-set,
and does not vanish on $\SM^1$.

Using the solution to the Dirichlet problem, we find a harmonic function
$\phi:\D\to \R$ such that $\phi=\log(F)$ on $\SM^1$. Replacing $F$ with $e^{-\phi}F$
(whose logarithm is clearly harmonic outside its zero set), we can thus assume
that $F$ equals $1$ on $\SM^1$.

We now recall that for every harmonic function defined on an annulus
$C(r_1,r_2):=\{z\ |\ r_1\le |z|\le r_2\}$,
its mean values along the concentric circles $|z|=r$ have a special behavior.

\begin{lemma}\label{lemtw}
Assume that $f:C(r_1,r_2)\to \R$ is harmonic. Then there exist real constants $a$, $b$ such that 
\[\int_{C(r)}f\,dl=r(\mu\log(r)+\nu)\]
for every $r\in[r_1,r_2]$, where $dl$ denotes the length element. We call $\mu$
the {\em virtual measure} of $f$ and denote it by $\mu(f)$. If $f$ extends to a harmonic function
on the disk $\{|z|\le r_2\}$, then its virtual measure vanishes.
\end{lemma}
\begin{proof}
Let us denote
\[K(r):=r^{-1}\int_{C(r)}f\,dl=\int_0^{2\pi}f(r\cos t,r\sin t)dt.\]
Then
\be\begin{split}\label{tw}
K'(r):=&\int_0^{2\pi}[\partial_x f(r\cos t,r\sin t)\cos t+\partial_y f(r\cos t,r\sin t)\sin t]dt\\
=&r^{-1} \int_{C(r)}\partial_x fdy-\partial_y fdx.
\end{split}\ee

Using this and the Green-Riemann theorem on $C(r_1,r_2)$ we get
\begin{align*} 0=&\int_{C(r_1,r_2)}\Delta_0(f)dxdy=\int_{C(r_2)}\frac{\partial f}{\partial
y}dx-\frac{\partial f}{\partial x}dy
-\int_{C(r_1)}\frac{\partial f}{\partial y}dx-\frac{\partial f}{\partial x}dy\\
=&r_1K'(r_1)-r_2K'(r_2).
\end{align*}
This shows that there exists a constant $\mu$ such that $rK'(r)=\mu$, thus proving the first claim.

If $f$ is defined on the whole disk, then $K(r)$ is bounded as $r$ tends to $0$, so necessarily $\mu(f)=0$.
\end{proof}

Notice that the virtual measure defined in Lemma \ref{lemtw} is \emph{additive}: $\mu(f_1+f_2)=\mu(f_1)+\mu(f_2)$.

Returning to our log-harmonic function $F$ and denoting $f:=\log(F)$, we shall exploit the fact
that $f$ is harmonic on some annulus $C(r_1,1)$
and vanishes on the outer circle $C(1)$.

\begin{lemma} \label{tw1}
Let $F:\overline\D\to [0,\infty)$ be a smooth log-harmonic function with at least one non-isolated
zero in $\D$ and identically equal to $1$ on $\SM^1$. 
Then the virtual measure of $f=\log(F)$ is positive.
\end{lemma}
\begin{proof}
We apply \eqref{lh} and the Green-Riemann formula on the disk $\D$ to get
\[0< 2\int_\D|dF|^2dx\wedge dy=\int_\D(|dF|^2-F\Delta_0(F))dx\wedge dy=
\int_{\SM^1}F\frac{\partial F}{\partial x}dy-F\frac{\partial F}{\partial y}dx.\]
Using \eqref{tw} and the fact that $F\equiv 1$ on $\SM^1$, the right hand term reads
\[\int_{\SM^1}F\frac{\partial F}{\partial x}dy-F\frac{\partial F}{\partial y}dx=
\int_{\SM^1}\frac{\partial f}{\partial x}dy-\frac{\partial f}{\partial y}dx=\mu(f),\]
so the virtual measure of $f$ is positive.
\end{proof}

Let $z_0\in \D$ be a non-isolated zero of $F$. By composing with an element of $\mathrm{Aut}(\D)$
if necessary, we can assume $z_0=0$. 
The virtual measure of the function $\log|z|$ is by direct computation equal to $2\pi$. For every
positive integer $n$, the function
$F_n(z):=|z|^{-n}F(z)$ is smooth by Lemma \ref{s}. The logarithm $f_n:=\log(F_n)=f-n\log|z|$ is
clearly harmonic on its domain of definition, and the restriction of $f_n$ to $\SM^1$ vanishes. We
can thus apply Lemma \ref{tw1} to $f_n$ 
and deduce that $\mu(f_n)>0$. On the other hand the virtual measure is additive, so
$\mu(f_n)=\mu(f)-2\pi n$ is negative
for $n$ large enough. This contradiction shows that $F$ does not have any non-isolated zeros, and proves the theorem.
\end{proof}

\begin{proof}[Proof of Theorem \ref{th2}] By Lemma \ref{l2} and Theorem \ref{is}, for every
$\alpha\in \Omega$ there exists an open disk $U_\alpha\ni\alpha$ and a holomorphic function
$h_\alpha:U_\alpha\to\C$ with $|F|=|h_\alpha|^2$. Lemma \ref{n} shows that for every $\alpha$ and
$\beta$ there exists a unique $A_{\alpha\beta}\in\SM^1$ with $h_\alpha=A_{\alpha\beta}h_\beta$ on
$U_\alpha\cap U_\beta$. The \v Cech cocycle $(A_{\alpha\beta})$  must be exact since
$\pi_1(\Omega)=1$. Thus $A_{\alpha\beta}=A_\alpha^{-1}A_\beta$ for some $A_\alpha\in \SM^1$, and so
$A_\alpha h_\alpha$ agree on intersections, thus defining a global solution $h$ on $\Omega$
satisfying $|F|=|h|^2$. \end{proof}

\section{Local embedding of Ricci metrics}

This section is devoted to the
\begin{proof}[Proof of Theorem \ref{main}]
Every point in $M$ has a neighborhood where the metric $g$ can be written as $g=e^{-2f}g_0$,
where $g_0$ is flat and $f$ is smooth. By \eqref{1} and \eqref{3}, $\Delta f=-K$. By Lemma
\ref{equiv}, the Ricci condition \eqref{R} implies $\Delta(\log|e^{-4f}K|)=0$ at points where $K$
does not vanish, in other words $e^{-4f}K$ is log-harmonic.
Theorem \ref{is} implies that if $K$ does not vanish identically, then it has only isolated zeros
and does not change sign on $M$.

\subsection*{Case 1. $K\leq 0$ on $M$} Let $P$ be an arbitrary point of $M$.
Choose a neighborhood $\D\ni P$ such that $K<0$ on
$\D\setminus\{P\}$. We can identify $(\D,g_0)$ with a disk in $\cz$ endowed with the Euclidean
metric $|dz|^2$ so that $P$ corresponds to $0$.

\subsection*{Case 1.1. $K\leq 0$, $K(0)\ne 0$} This was originally treated by Ricci-Curbastro
\cite{ricci}, we give here an argument in our framework.
By Lemma \ref{equiv}, the Ricci condition \eqref{R} implies that
the metric $g_r:=(-K)^rg$ is flat for $r=\tfrac12$ and has constant Gaussian curvature equal to
1 for $r=1$.
Consequently, by shrinking $\D$ if necessary, we may assume that there exist isometries
\begin{align*}
&\f:(\D,g_{1/2})\to (U_0,|dz|^2), &U_0\subset \C\\
\intertext{and}
&\psi:(\D,g_1)\to\left(U_1,\frac{4|dz|^2}{(1+|z|^2)^2}\right),&U_1\subset\C.
\end{align*}
The maps $\f$ and $\psi$ are holomorphic functions of $z$, so we can write
\begin{align*}
g_{1/2}=\sqrt{-K}e^{-2f}|dz|^2=|\f'|^2|dz|^2,&&
g_1=(-K)e^{-2f}|dz|^2=\frac{4|\psi'|^2|dz|^2}{(1+|\psi|^2)^2}
\end{align*}
whence
\[e^{-f}=(1+|\psi|^2)\frac{|\f'|^2}{2|\psi'|}.\]
Since $\psi'$ does not vanish on $\D$, there exists a holomorphic map $\zeta:\D\to \C$ with $\zeta^2=2\psi'$.
Thus $e^{-f}=|a|^2+|b|^2$ for holomorphic functions $a:=\tfrac{\f'}\zeta$ and $b:=\tfrac{\psi\f'}{\zeta}$,
so by Corollary \ref{l6}, $(\D,g)$ has an isometric minimal embedding in $\R^3$.

\subsection*{Case 1.2. $K\leq 0$, $K(0)=0$}
Using Case 1.1 treated above, for every point $\alpha\in \D\setminus \{0\}$ there exists an open
disk $U_\alpha\subset \D\setminus \{0\}$ containing $\alpha$ and a holomorphic function
$g_\alpha:U_\alpha\to\C^2\setminus\{0\}$
such that $e^{-f}=|g_\alpha|^2$ on $U_\alpha$. 
Moreover, since $K$ does not vanish on $U_\alpha$, we have $\Delta_0\log(|g_\alpha|^2)\ne 0$.
By Lemma \ref{l3}, there exist unique matrices $A_{\alpha\beta}\in\U(2)$
with $g_\alpha=A_{\alpha\beta} g_\beta$ on $U_\alpha\cap U_\beta$, which clearly form a \v Cech cocycle. 

Consider the universal cover $\widetilde{\D^*}=\{z\in\cz;\Re(z)<0\}$ of $\D^*$ and the projection
$p:\widetilde{\D^*}\to\D^*$ given by the exponential map. We denote $V_\alpha:=p^{-1}(U_\alpha)$
and $G_\alpha(z):=g_\alpha(e^z)$.
Since $\check H^1(\widetilde\D^*;\U(2))=0$, the \v Cech cocycle $(V_\alpha,A_{\alpha\beta})$ is
exact, so there exist
locally constant functions $A_\alpha:V_\alpha\to\U(2)$ with $A_{\alpha\beta}=A_\alpha^{-1} A_\beta$
on $V_\alpha\cap V_\beta$.
This shows the existence of a global holomorphic map $G:\widetilde\D^*\to\C^2$ (given by
$G=A_\alpha G_\alpha$ on $V_\alpha$),
with $e^{-f(e^z)}=|G(z)|^2$ for all $z\in \widetilde{\D^*}$.
By Lemma \ref{n}, there exists $A\in\U(2)$ such that $G(z+2\pi i)=AG(z)$. We diagonalize $A=
P\begin{bmatrix}e^{2\pi i\theta_1}&0\\0&e^{2\pi i\theta_2}\end{bmatrix}P^{-1}$ 
for $P\in\U(2)$, $\theta_1,\theta_2\in[0,1)$. 

The map
$$H(z):=\begin{bmatrix}e^{-z\theta_1}&0\\0&e^{-z\theta_2}\end{bmatrix}P^{-1}G(z)$$
is invariant by translation with $2\pi i$, hence it descends to a holomorphic map
$h=(h_1,h_2):\D^*\to\C^2$ with
$h(e^z)=H(z)$. Setting $w:=e^z$ we get
\[e^{-f(w)}=|G(z)|^2=\left|\begin{bmatrix}e^{z\theta_1}&0\\0&e^{z\theta_2}\end{bmatrix}h(w)
\right|^2=|h_1(w)|^2|w|^{2\theta_1}+|h_2(w)|^2|w|^{2\theta_2}.\]
Let $k_j$ be the vanishing order of $h_j$ at $0$. One has $h_j(w)=l_j(w)w^{k_j}$ with $l_j$
holomorphic and $l_j(0)\ne 0$.
We thus have
\[e^{-f(w)}=l_1(w)|w|^{r_1}+l_2(w)|w|^{r_2}\]
where $l_i$ are smooth functions which do not vanish near 0 and $r_j=2k_j+2\theta_j$. 
Such a function is smooth if and only if $r_1$ and $r_2$ are both even integers, which implies $\theta_j=0$, so
$e^{-f}=|h|^2$ and the conclusion follows from Corollary \ref{l6}.

\subsection*{Case 2. $K\ge 0$ on $M$}
\subsection*{Case 2.1. $K\ge 0$, $K(0)\neq 0$} For $r\in\rz$ let $g_r:=K^r g$.
Then as in Case 1.1, using Lemma \ref{kp} we find holomorphic maps $\f,\psi$ from $\D'\subset \D$ to $\cz$,
respectively to $\hz^2=\D$, satisfying
\[\begin{cases}
   g_{1/2}:=\sqrt{K}e^{-2f}|dz|^2=|\f'|^2|dz|^2,\\
   g_1:=Ke^{-2f}|dz|^2=\frac{4|\psi'|^2|dz|^2}{(1-|\psi|^2)^2}.
  \end{cases}
\]
It follows that $e^{-f}=|a|^2-|b|^2$ for holomorphic functions $a:=\tfrac{\f'}\zeta$ and
$b:=\tfrac{\psi\f'}{\zeta}$, for some square root $\zeta$ of $2\psi'$.
By Lemma \ref{l61}, a neighborhood of $0$ in the disk $(\D',g)$ has an isometric maximal embedding in $\R^{2,1}$.

\subsection{Case 2.2. $K\geq 0$, $K(0)=0$}
The proof is more involved than in Case 1.2, essentially because the group $\U(1,1)$ of isometries
of the indefinite Hermitian form $|\cdot|_-$ is non-compact. Using Case 2.1 we obtain like before a
holomorphic map 
$G:\widetilde{\D^*}\to\cz^2$ with $|G(z)|^2_-=e^{-f(e^z)}$. By the second part of Lemma \ref{l3},
$G(z+2\pi i)=AG(z)$ for some matrix $A\in \U(1,1)$. We wish to show that $A=1$, and then that 
$G$ descends to a map from $\D^*\to \cz^2$ which extends holomorphically to $\D$.

Every element $A$ of $\U(1,1)$ is conjugated (inside $\U(1,1)$) to a matrix of the form $e^{2\pi
i\theta} B$ with $\theta\in[0,1)$ and $B$ one of
\begin{align*}
A_1=\begin{bmatrix} e^{2\pi i\alpha}&0\\0&e^{-2\pi i\alpha}\end{bmatrix},&&
A_2=\begin{bmatrix} 1+2\pi ia&2\pi a\\2\pi a&1-2\pi ia\end{bmatrix},&&
A_3=\begin{bmatrix} \cosh (2\pi t) &-\sinh (2\pi t)\\ -\sinh (2\pi t)&\cosh (2\pi t)\end{bmatrix}
\end{align*}
for some real constants $\theta,\alpha,a,t$. The three cases occur according to whether $|\tr(A)|$
is smaller, equal or larger than $2$. Consider the group morphisms
$B_j:(\cz,+)\to{\rm GL}_2(\cz)$ defined by
\begin{align*}
B_1(z)=\begin{bmatrix} e^{z\alpha}&0\\0&e^{-z\alpha}\end{bmatrix},&&
B_2(z)=\begin{bmatrix} 1+za&za\\za&1-za\end{bmatrix},&&
B_3(z)=\begin{bmatrix} \cosh (izt) &\sinh (izt)\\ \sinh (izt)&\cosh (izt)\end{bmatrix}.
\end{align*}
We clearly have $B_j(2\pi i)=A_j$, $j=1,2,3$. It follows that, if $A=Pe^{2\pi i\theta}A_jP^{-1}$, then
\[H(z):=e^{-z\theta}B_j(z)^{-1} PG(z)\]
is invariant by the translation with $2\pi i$, hence it descends to a map $h:\D^*\to\cz^2$
satisfying $h(e^z)=H(z)$. Since $B_j$ depends holomorphically on $z$, the map $h$ is also
holomorphic. Let $w=e^z$, $z=x+iy$ and $r:=|w|$. Denoting $B_j=(b_{kl})_{k,l=1}^2$ we have
\begin{align*}
e^{-f(w)}=&|G(z)|^2_-=|e^{z\theta} B_j(z) h(w)|^2_-
= r^{2\theta} |B_j(z) h(w)|^2_-\\
= & r^{2\theta} ((|b_{11}|^2-|b_{12}|^2)|h_1|^2 -(|b_{22}|^2-|b_{21}|^2)|h_2|^2
+2\Re((b_{11}\bar{b}_{12}-b_{21}\bar{b}_{22})h_1\bar{h}_2)).
\end{align*}
In each of the three cases we compute
\[r^{-2\theta}e^{-f(w)}=\begin{cases}
r^{2\alpha} |h_1|^2 - r^{-2\alpha} |h_2|^2, & j=1,\\
|h_1|^2 (1+2a\log r) +|h_2|^2 (-1+2a\log r) +2\Im(h_1 \bar{h}_2) ay,& j=2,\\
 (|h|^2_- \cos(2t\log r)+2\Im(h_1\bar{h}_2)\sin(2t\log r)),& j=3.
\end{cases}\]
For $j=1$ it is clear from the Picard theorem that $h_1$ cannot have a essential singularity at
$0$, and then the same reasoning applies to $h_2$ to deduce that $h$ is meromorphic in $0$. Then
since $e^{-f}$ is smooth, it follows that $\alpha\in \zz$, hence $A=e^{2\pi i\theta}I_2$.

For $j=2$, the right-hand side must be $2\pi$-periodic in $y$ so if $a\neq 0$ then $h_1,h_2$ are
proportional, which would imply that the curvature vanishes identically. Hence $a=0$ and so $A=I_2$.

For $j=3$ take $r_k=e^{-k\pi/t}$ for $k\in\N$. On the circles of radii $r_k\to 0$, the function
$|h_1|$ is uniformly bounded from below since $\theta\geq 0$. By the maximum principle, it must be
bounded from below in a neighborhood of $0$ and hence $h_1$ is meromorphic in $0$. Picard's theorem
again shows that $h_2$ is meromorphic at $0$. With a little more effort one sees that $t$ must be
$0$.

In all three cases we have obtained $A=e^{2\pi i\theta}I_2$ and $h$ meromorphic at $0$. In order for
$r^{2\theta} |h|^2_-$ to be smooth, it is necessary that $\theta=0$ (we cannot have $|h|^2_-=0$
since this would entail the vanishing of the Gaussian curvature $K$). Then clearly $h$ is
holomorphic at $0$, so Lemma \ref{l61} ends the proof.
\end{proof}

\section{Compact Ricci surfaces}\label{ss}

In this section we study compact Ricci surfaces without boundary. From
Theorem \ref{main}, for any such surface, the Gaussian curvature
$K$ does not change sign on $M$, so integrating \eqref{R} over $M$ we see that $K$ has to be
non-positive. In the non-negative curvature case we enlarge therefore the class of compact Ricci
surfaces by allowing conical singularities. Our examples of compact Ricci surfaces stem from three
main sources: triply periodic
surfaces, branched coverings of $\SM^2$, and spherical manifolds with conical singularities.

\subsection{Triply periodic minimal surfaces} A complete minimal surface $S\subset\R^3$ is called triply
periodic if it is invariant under the translation group defined by a lattice $\Lambda\subset\R^3$.
By Lemma \ref{l7}, the quotient $M:=S/\Lambda$ is a compact Ricci surface.

Triply periodic minimal surfaces in $\R^3$ are abundant in the literature. The first five examples
were constructed by Schwarz at the end of the 19$^{\rm{th}}$ century. Later on, in his 1970 NASA technical
report \cite{schoen} (see also \cite{karcher}), Schoen constructed 17 new examples of such surfaces.
A significant number of papers appeared since then on this subject, a partial account of which
can be found in \cite{meeks}.
Recently Traizet \cite{traizet} proved that for every lattice $\Lambda\subset\R^3$ and for every
$g\ge 3$, there exists a minimal surface $S$ in $\R^3$ invariant by $\Lambda$ such that $S/\Lambda$
has genus $g$. In particular, this shows the existence of compact Ricci surfaces in any genus $g\ge
3$.

Recall now that for every minimal surface $S\subset \R^3$, the Gauss map $G:S\to\SM^2$ is a branched
covering whose branching points are precisely the zeros of the Gaussian curvature of $S$ (see \cite{meeks},
Proposition 2.1 and Corollary 2.1). Consequently, if $S$ is triply periodic, the compact Ricci
surface $M:=S/\Lambda$ is a branched covering of $\SM^2$ too.

Note that the compact Ricci surfaces obtained in this way are branched coverings of
$\SM^2$ with $n$ sheets and have genus $g=n+1\ge 3$ (\cite[Thm. 3.1]{meeks}).

\subsection{Spherical surfaces with conical singularities}

We have seen in Lemma \ref{equiv} that the metric $g_1:=(-K)g$ is locally isometric to $\SM^2$ and
the metric $g_{1/2}:=\sqrt{-K}g$ is flat for every Ricci surface $(M,g)$ with non-positive Gaussian
curvature $K$. Of course, the metrics $g_{1/2}$ and $g_1$ have (conical) singularities at points
where $K$ vanishes. This suggests the idea of constructing a
flat metric $g_{1/2}$ with conical singularities on a given Riemann surface $M$, then a spherical metric
$g_{1}=Vg_{1/2}$ with conical singularities in the same conformal class, and then use Lemma
\ref{trfsr} to show that the metric $g:=V^{-1}g_{1/2}$ is a Ricci metric.

\begin{lemma}\label{lemfl}
Let $(M,J)$ be a Riemann surface, $\cP\subset M$ a discrete set and $\beta:\cP\to \R$ a function. In case
$M$ is closed, assume that 
\[\sum_{P\in\cP} (\beta(P)-2\pi)=-2\pi\chi(M).\]
Let $z$ be a complex coordinate on $M$ near $\cP$.
Then there exists a flat metric $g$ on $M\setminus\cP$ compatible with $J$ which near each $P\in\cP$ is of the form
\[g= e^{2v} |z|^{\frac{\beta(P)}{\pi}-2}|dz|^2\]
for some $v\in\cun(M,\R)$.
\end{lemma}
\begin{proof}
Consider a metric $h$ in the conformal class of $M$ ({\em i.e.}, compatible with $J$), such that $h=|dz|^2$
near $\cP$.
Let $u$ be a smooth positive function on $M\setminus\cP$ which equals $|z|^{\frac{\beta(P_j)}{\pi}-2}$ near every $P_j\in\cP$.
Since $\Delta_h \log u$ vanishes near $\cP$, it extends to a smooth function on $M$. We try to solve the Laplace equation
\begin{align}\label{ecL}
\Delta_h v + K_h+ \tfrac12 \Delta_h \log u=0
\end{align}
with $v\in\cun(M,\rz)$. Like every elliptic equation with the unique
continuation property, \eqref{ecL} can be solved inside $\cun$ functions on any non-compact manifold
\cite[Thm.\ 5, p.\ 341]{malgrange}.
When $M$ is a closed surface, the equation
$\Delta v=H$ has solutions if and only if $H$ has zero mean. 
The set $\cP$ is finite, $\cP=\{P_1,\ldots,P_k\}$. By Gauss-Bonnet and 
\cite[Lemma 4]{uscs}, the integral of $K_{h}+ \frac12 \Delta_h \log u$
equals 
\[\int_M (K_{h} + \tfrac12 \Delta_h \log u)\vol_{h}
=2\pi\chi(M) +\pi\sum_{j=1}^k\left(\frac{\beta(P_j)}{\pi}-2\right),\]
which vanishes precisely when \eqref{sumal} holds. Let therefore $v$ be a solution to \eqref{ecL}.
From \eqref{3}, the metric
$g:=e^{2v}uh$ is flat, and by construction near each $P_j\in\cP$ it takes the desired form.
\end{proof}
The above result is due to Troyanov \cite{troyanov1} in the case where $M$ is closed, see also \cite{uscs}.

We define a conical spherical metric on $M$ to be a metric $g_1$
of curvature $1$ outside an isolated set  $\cP$ which in some holomorphic coordinate $z$ near each $P_j\in\cP$ takes the form
\begin{align}\label{g1sph}
g_1=\frac{4n_j^2 |z|^{2n_j-2}|dz|^2}{(1+|z|^{2n_j})^2}
\end{align}
for some $n_j\in(0,\infty)$. The number $\alpha_j=2\pi n_j$ is the cone angle at $P_j$. This
definition makes sense for real $n_j$ but for us it will be useful for $n_j\in\N^*$.

\begin{prop}\label{propfsr}
Let $(M,g_1)$ be a Riemannian surface with a spherical metric with conical singularities of angles
$\alpha_j=2\pi n_j$ with
$n_j\geq 2,n_j\in\zz$ at each $P_j\in\cP\subset M$. If $M$ is closed,
assume additionally that the conical angles $\alpha_1,\ldots,\alpha_k$
at the conical points $P_1,\ldots,P_k$ satisfy
\begin{equation}\label{sumal}
\sum_{j=1}^k (\alpha_j-2\pi) = -4\pi\chi(M).\end{equation}
where $\chi(M)$ is the Euler characteristic $\chi(M)$. Then $M$ admits a Ricci metric in the conformal class of $g_1$.
\end{prop}
\begin{proof}
Near every conical point of angle $\alpha_j$, there exists a complex parameter $z$ with respect to
which the spherical metric takes the form \eqref{g1sph}. The function
$P_j\mapsto\beta(P_j):=\frac{\alpha_j}{2}+\pi$ satisfies the hypothesis of Lemma \ref{lemfl} if and
only if
\eqref{sumal} holds. From Lemma \ref{lemfl}, there exists on $M\setminus\cP$ a flat metric
$g_{1/2}$ conformal to $g_1$ which near $P_j\in\cP$ is of the form
\[g_{1/2}= e^{2v}|z|^{n_j-1}|dz|^2\]
for some smooth $v\in\cun(M,\R)$. Let $V$ be the conformal factor
defined by $g_1=Vg_{1/2}$. Near a conical point, $V$ equals
\[V= e^{-2v} \frac{4n_j^2 |z|^{n_j-1}}{(1+|z|^{2n_j})^2}.\]
By hypothesis, $n_j\geq 2, n_j\in\zz$.
Hence $V$ vanishes precisely at the conical points, and the metric $g:=V^{-1}g_{1/2}$ is smooth on
$M$, including at the points $P_j$ where it reads
\[g=\tfrac{1}{4}n_j^{-2}e^{4v}(1+|z|^{2n_j})^{-2}|dz|^2.\]
By Lemma \ref{trfsr}, $g$ satisfies the Ricci condition outside the isolated zeros of $V$ and so it
is a Ricci metric on $M$.
\end{proof}

Using this result, we can give more examples of Ricci metrics on compact Riemann surfaces.
\begin{cor}\label{corric}
Let $M$ be a compact Riemann surface of genus $g$, and $\phi:M\to \SM^2$ a branched cover of degree
$n=g-1$. Then $M$ admits Ricci metrics.
\end{cor}
\begin{proof}
Pull back the spherical metric from $\SM^2$ to $M$ via $\phi$, {\em i.e.},
$g_1:=\phi^*g_{\text{sph}}$. Every branching point of order $n_j$ becomes a conical point of
$(M,g_1)$ of angle $2\pi n_j$. By the Riemann-Hurwitz formula,
\[-\sum_{j=1}^k (n_j-1)+2n=2-2g\]
hence the conical angles of $g_1$ satisfy the constraint \eqref{sumal} if and only if $n=g-1$.
It follows from Proposition \ref{propfsr} that for covers of this degree, the surface $M$ admits
Ricci metrics.
\end{proof}

By composing $\phi$ with a conformal transformation of $\SM^2$ which is not an isometry (an element
in $\text{PSL}_2(\cz)\setminus \text{SO}_3$) we obtain another Ricci metric, hence Ricci metrics
arising from branched coverings are not unique in their conformal class.

Generically, a surface $M$ of genus $g$ does not admit branched coverings over $\SM^2$ of degree
$n\leq g-1$ with a branching point of order $n$, \emph{cf.}\  \cite{farkas}.

\begin{example}
Let $M$ be a hyperelliptic Riemann surface of odd genus. Then $M$ admits Ricci metrics.
Indeed, if $\phi:M\to \SM^2$ is a branched double cover, then $\phi^{(g-1)/2}$ is a branched
cover of degree $g-1$ and we can apply Corollary \ref{corric}.
\end{example}

\subsection{An explicit Ricci metric with one zero for the Gauss curvature in every genus $g\geq 2$}
We give below a different way of constructing compact Ricci surfaces of every genus $g\ge 2$,
which shows that there are definitely more Ricci surfaces than triply periodic minimal surfaces.

Let $M$ be a closed oriented topological surface of genus $g\geq 1$. Fix a homology basis
consisting of $2g$ 
simple closed curves $\alpha_1,\ldots,\alpha_{2g}$ such that $\alpha_j$ is disjoint from $\alpha_i$
unless $\{i,j\}=\{2k-1,2k\}$ for some $k\in\{1,\ldots,g\}$, and $\alpha_{2k-1}$ meets $\alpha_{2k}$
in precisely one point. Choose a point $p\in M$ and choose simple loops $\gamma_j$ freely homotopic
to $\alpha_j$ such that they meet only in $p$. By cutting along $\gamma_j$, we obtain a $4g$-gon
$Q$ with vertices $P_1,\ldots,P_{4g}$. 
To recover $M$, one must identify in $Q$ the pairs of sides $\gamma_j'$ and $\gamma_{j}''$
corresponding to the cut along $\gamma_j$. 

By joining $P_1$ with $P_3,\ldots,P_{4g-1}$ we obtain a (combinatorial) decomposition of $Q$ 
into triangles. To define a spherical metric on $M$ it is enough to endow each of these triangles
with the structure of a spherical triangle with geodesic sides, 
and then glue them in the obvious way provided that the lengths of $\gamma_j'$ and $\gamma_{j}''$
coincide. A basic remark is that the result of such a gluing is a smooth spherical
metric along the interiors of the edges. In the unique vertex $P$, we get a conical point of total 
angle equal to sum of the angles of the $4g-2$ triangles. We get moreover a conformal structure on 
$M$, since the singularity of the conformal structure at $P$ is removable.

\begin{example}\label{exf}
For every $g\geq 2$ and for every $\pi(4g-2)<\theta< 5\pi(4g-2)$ there exists at least one
spherical 
metric on a surface of genus $g$ with a unique conical point of angle $\theta$.
We construct it by requiring the $4g-2$ triangles in the $4g$-gon $Q$ to be equilateral (and
congruent) of angle $\alpha=\frac{\theta}{3(4g-2)}$
(an equilateral spherical triangle of angle $\alpha$ exists for every
$\alpha\in(\pi/3,5\pi/3)$). In particular, by choosing $\alpha=\pi\frac{4g-3}{6g-3}$, the conical
angle becomes $2\pi(4g-3)$, and so the hypothesis of Proposition \ref{propfsr} holds.
\end{example}

In general, there are $12g-6$ edges which must be identified in pairs, hence $6g-3$ parameters
giving the lengths 
of the edges. In each triangle, the edges $e_1,e_2,e_3$ must satisfy a spherical triangle inequality of the form
\[e_1+e_2 > \min\{e_3,2\pi-e_3\}.\]
We want to 
prescribe the conical angle at $P$ to be equal to $2\pi(4g-3)$. There seem therefore to be $6g-4$
degrees of 
freedom for this construction. This coincides with the dimension of the total space of the
tautological fibration over the Teichm\"uller space of $M$, which is a surface fibration of fiber 
$(M,c)$ over the conformal structure $c$. Fixing the conical point
$P$ amounts to choosing a point in the fiber. So we conjecture that in every conformal class on $M$
and for every point $P\in M$ there exists a spherical metric on $M$ with a conical singularity at $P$
of angle $2\pi(4g-3)$.

An existence result for spherical conical metrics was proved by Troyanov \cite{troyanov2}, but it
does not cover the case needed here. Indeed, when there exists a unique conical point, Theorem
C in \cite{troyanov2} requires the angle to be comprised strictly between $\pi(4g-2)$ and $\pi(4g+2)$.
The upper bound is due to the explicit Trudinger constant $4\pi$ in the Trudinger-Sobolev inequalities.
Thus we cannot so far prove that in every conformal class there exist spherical metrics, but we can
at least construct one Ricci surface in every genus $g\geq 2$ with curvature vanishing at precisely
one point.

\begin{theorem}
For every  $g\geq 2$ there exists an oriented closed surface of genus $g$ with
a Ricci metric whose curvature vanishes precisely at one point, to order $8{g-1}$.
\end{theorem}
\begin{proof}
Apply Proposition \ref{propfsr} to the spherical metric on $M$ with one conical point of angle
$2\pi(4g-3)$ constructed in Example \ref{exf}.
\end{proof}

\subsection{Conical Ricci metrics of positive curvature}

A metric on the unit disk $\D$ is called \emph{conical} at $z_0\in\D$ of angle $\alpha\in\rz$ if it is of the form
$g=|z-z_0|^{\frac {\alpha}{\pi}-2} h$ where $h$ is a smooth conformal metric. This definition
extends directly to Riemann surfaces.

\begin{theorem}
Let $M$ be a Riemann surface of genus $g\geq 2$, $P_1,\ldots,P_k\in M$ marked points and
$\alpha_j=2\pi n_j$ prescribed real angles satisfying
\begin{equation} \label{ecu}
\sum_{j=1}^k (\alpha_j-2\pi)=2\pi(4g-4).
\end{equation}
Then there exists a positively curved Ricci metric on $M\setminus\{P_1,\ldots,P_k\}$ with conical
singularity of angle $\alpha_j$ at $P_j$ for all $j=1,\ldots,k$.
\end{theorem}
\begin{proof}
Let $g_{-1}$ be the unique smooth hyperbolic metric in the conformal class of $M$ given by the
Riemann uniformization theorem. The hypothesis \eqref{ecu} on the angles implies 
\[\sum_{j=1}^k\left(\pi+\frac{\alpha_j}{2}-2\pi\right)=2\pi(2g-2),\] 
so Lemma \ref{lemfl} gives us a conical flat metric $g_0$ on $M$ with cone angle
$\pi+\frac{\alpha_j}{2}$ at $P_j$. Let $V$ be the conformal factor such that $g_0=Vg_{-1}$. Then by
Lemma \ref{hfr} the metric defined by $g_R:=Vg_0$
is Ricci outside the conical points, with positive Gaussian curvature $K=V^2$. Near $P_j$, $V$ is
by construction of the form $|z|^{\frac{\alpha_j}{2\pi}-1}$ times a smooth function on $M$, so
the metric $g_R=V^2g_{-1}$ is conical on $M$ in the sense of our definition, of angle $\alpha_j$ at
$P_j$.
\end{proof}

More generally, in a given conformal class with marked points, there exist unique hyperbolic 
metrics of prescribed conical singularities (see \cite[Theorem A]{troyanov2}). 
The condition \eqref{ecu} in the above theorem can therefore be relaxed. 
By the same argument, we can construct conical Ricci metrics of non-positive curvature.
However we do not have a definitive answer to the uniqueness question,
so we leave open the classification of conical Ricci metrics.

\appendix

\section{Link with the Weierstrass-Enneper parametrization}

We adopted in this paper the viewpoint of differential geometry. There exists an alternate local description of 
minimal surfaces, found by Enneper and Weiertrass, as being governed by $3$ holomorphic functions with certain additional properties. 
In this appendix we show how to translate some of our preliminary results
in the language of the Weierstrass-Enneper parametrization. 

Let $A:\Omega\to \R^3$ be an isothermal parametrization of a surface $(M,g)\subset \R^3$. This means that
the vector fields $A_x:=\partial_x, A_y:=\partial_y$ are mutually orthogonal and of equal length:
\begin{align}\label{isoparA}
|A_x|=|A_y|=e^{-f}, &&\langle A_x,A_y\rangle =0
\end{align}
and so the (pull-back by $A$ of the) metric on $M$ inherited from $\R^3$ is given by $g=e^{-2f}|dz|^2$. The second fundamental form is computed in terms of the unit-length normal 
field $\nu=e^{2f} A_x\times A_y$:
\[\langle W(X),Y\rangle = \langle X(Y),\nu\rangle\]
for every vector fields $X,Y$ tangent to $M$. 
In the basis $\px,\py$, 
\[ W=e^{2f} \begin{bmatrix}
\langle A_{xx},\nu \rangle &\langle A_{xy},\nu \rangle\\
\langle A_{yx},\nu \rangle &\langle A_{yy},\nu \rangle
\end{bmatrix} =
e^{4f} \begin{bmatrix}
\langle A_{xx},A_x\times A_y \rangle &\langle A_{xy},A_x\times A_y \rangle\\
\langle A_{yx},A_x\times A_y \rangle &\langle A_{yy},A_x\times A_y \rangle
\end{bmatrix}.
\]
We deduce 
\begin{align*}
\tr(W)=e^{2f} \left(\langle W(A_x),A_x \rangle +\langle W(A_y),A_y \rangle\right)=e^{2f} \langle A_{xx}+A_{yy},\nu \rangle .
\end{align*}
Notice that the tangential component of $A_{xx}+A_{yy}$ vanishes
(we compute $\nabla_{\px}\px=f_y\py-f_x\px=-\nabla_{\py}\py$ for the Levi-Civita connection on $M$).
Hence, $M$ is minimal if and only if $A$ is harmonic. From now on we assume this to be the case.

Let $C:=A_x-iA_y=2{\partial_z A}$. Since $A$ is harmonic, the $\C^3$-valued function $C$ must be holomorphic. Moreover, 
if $\langle \cdot, \cdot\rangle$ denotes the $\C$-bilinear extension of the standard scalar product on $\R^3$, the identities
\eqref{isoparA} encoding the fact that
$A$ is an isothermal parametrization mean precisely 
\begin{align*}
\langle C, C\rangle=0, &&\langle C, \overline{C}\rangle=2e^{-2f}.
\end{align*}
Define a complex-valued function from 
the coefficients of $W$:
\[
h:= e^{-2f}(W_{11}-iW_{12})=e^{2f}\langle A_{xx}-iA_{xy},A_x\times A_y\rangle.
\]
\begin{lemma}
The function $h$ is holomorphic.
\end{lemma}
\begin{proof}
We can re-write $h$ as 
\begin{align}
 h={}&e^{2f}\langle C', A_x\times A_y\rangle = e^{2f}\langle C', C\times A_y\rangle = e^{2f}\langle C', C\times \tfrac{1}{i}({\partial}_{\bar z} A-\partial_z A)\rangle\nonumber\\
={}&\tfrac{1}{2i}e^{2f}\langle C', C\times \overline{C}\rangle=\tfrac{1}{2i}e^{2f}\langle C'\times C, \overline{C}\rangle.
\label{formulah}
\end{align}
Let us show that ${\partial}_{\bar z}h=0$. Since $C$ and $C'$ are holomorphic, it is enough to show
\begin{align}\label{orthol}
\langle C'\times C, {\partial}_{\bar z}(e^{2f}\overline{C})\rangle=0.
\end{align}
For this, note the orthogonality relations (always with respect to the complexified inner product)
\begin{align}\label{orto12}
\langle C,C\rangle=0, &&\langle C,C'\rangle=0, 
\end{align}
the second one being deduced from the first by applying $\partial_z$. Also, from 
\begin{equation}\label{normc}
\langle C,\overline{C}\rangle = 2e^{-2f}
\end{equation}
we get, applying ${\partial}_{\bar z}$, 
\[\langle C,\overline{C'}\rangle=-4f_z e^{-2f}\]
and so 
\begin{align}\label{orto3}
\langle C,\overline{C'}+2f_z\overline {C}\rangle=0.
\end{align}
We have found three vectors ($C,C'$ and $\overline{C'}+2f_z\overline {C}$) orthogonal to $C$, they must therefore be linearly dependent since the complexified inner product is non-degenerate. This implies the vanishing \eqref{orthol}.
\end{proof}

Since $W$ is trace-free and symmetric
we get $\det(W)=-W_{11}^2-W_{12}^2=-e^{4f}|h|^2$. 
By the Gauss equation, the curvature of $g$ equals $K=\det(W)$.

On the other hand,
since $g=e^{-2f}|dz|^2$, we get $K=-e^{2f}\Delta f$. Therefore $e^{-2f}\Delta f$ equals the norm squared of the holomorphic function $h$.

We are now in position to compute $h$ in terms of the Weierstrass-Enneper representation of $A$, exploiting the fact that $C=2\partial_z A$ is holomorphic and isotropic for the complexified inner product. We assume that $A_x,A_y$ are linearly independent (since they are the tangent vector fields to $M$ in a chart).
Write $C=(a,b,c)$ with holomorphic components $a,b,c$. We claim that since $a^2+b^2+c^2=\langle C, C\rangle=0$, there exist holomorphic functions $\alpha,\beta$ such that
\begin{align}\label{WE}
a={}\alpha(1+\beta^2),&&b={}i\alpha(1-\beta^2),&& c={}2i\alpha\beta.
\end{align}
To this end, set $\alpha:=\frac{a-ib}{2}$ and $\beta= \frac{c}{ia+b}$. It is immediate (using $c^2=-(a+ib)(a-ib)$) that \eqref{WE} holds. 
Furthermore,
\begin{lemma}
The holomorphic function $h$ is given by 
\[ h=-2i\alpha \beta'.\]
\end{lemma}
\begin{proof}
We use the expression \eqref{formulah} for $h$. From \eqref{normc}, 
\begin{align*}
2e^{-2f}={}& |a|^2+|b|^2+|c|^2\\
={}& |\alpha|^2(|1+\beta^2|^2+|1-\beta^2|^2+4|\beta|^2)\\
={}&2|\alpha|^2(1+|\beta|^2)^2.
\end{align*}
Next, we write using determinants
\begin{align*}
\langle C', C\times \overline{C}\rangle={}&
\begin{vmatrix}
\alpha(1+\beta^2)& \overline{\alpha}(1+\overline{\beta}^2) & \alpha'(1+\beta^2)+2\alpha\beta'\beta\\
i\alpha(1-\beta^2)& -i\overline{\alpha}(1-\overline{\beta}^2) & i\alpha'(1-\beta^2)-2i\alpha\beta'\beta\\
2i\alpha\beta & -2i \overline{\alpha} \overline{\beta} & 2i\alpha'\beta+2i\alpha\beta'
\end{vmatrix}.
\end{align*}
In the third column, the first terms form a multiple (namely, $\alpha'/\alpha$ times) the first column, hence they do not contribute to the determinant.
We extract $\alpha$, $\overline\alpha$, resp.\ $\alpha\beta'$ which are common factors in the first, second, respectively third column. We also extract $i$, resp.\ $2i$ as common factors in the second, respectively third line. We are left with 
\begin{align*}
\langle C', C\times \overline{C}\rangle={}&-2|\alpha|^2\alpha\beta' 
\begin{vmatrix}
1+\beta^2& 1+\overline{\beta}^2& 2\beta\\
1-\beta^2& -(1-\overline{\beta}^2) & -2\beta\\
\beta &-\overline{\beta}&1
\end{vmatrix}.
\end{align*}
The above determinant yields (after adding the second line to the first for simplicity)
\[
2(-|\beta|^2-1-\overline{\beta}^2(1+\beta^2))=-2(1+|\beta|^2)^2.
\]
Gathering the above formulas we get the lemma.
\end{proof}

\bibliographystyle{amsplain}

\end{document}